\documentclass[12pt]{amsart}
\usepackage{amstext,amsfonts,amssymb,amscd,amsbsy,amsmath,verbatim, mathrsfs, fullpage}
\usepackage[alphabetic,abbrev,lite]{amsrefs} 
\usepackage{microtype}

\newcommand{\shrinkmargins}[1]{
  \addtolength{\textheight}{#1\topmargin}
  \addtolength{\textheight}{#1\topmargin}
  \addtolength{\textwidth}{#1\oddsidemargin}
  \addtolength{\textwidth}{#1\evensidemargin}
  \addtolength{\topmargin}{-#1\topmargin}
  \addtolength{\oddsidemargin}{-#1\oddsidemargin}
  \addtolength{\evensidemargin}{-#1\evensidemargin}
  }

\shrinkmargins{.7}

\newcommand{\field}[1]{\mathbb{#1}}

\newcommand{\Z}{\field{Z}}

\newcommand{\F}{\field{F}}

\newcommand{\R}{\field{R}}

\newcommand{\SSS} {\mathcal{S}}

\newcommand{\ra}{\rightarrow}

\newcommand{\beq}{\begin{displaymath}}
\newcommand{\eeq}{\end{displaymath}}
\newcommand{\beqn}{\begin{equation}}
\newcommand{\eeqn}{\end{equation}}

\theoremstyle{plain}
\newtheorem{thm}{Theorem}
\newtheorem{prop}[thm]{Proposition}
\newtheorem{cor}[thm]{Corollary}

\theoremstyle{definition}

\theoremstyle{remark}
\newtheorem{rem}[thm]{Remark}

\title{On large subsets of $\F_q^n$ with no three-term arithmetic progression}
\author{Jordan S. Ellenberg }
\address{Department of Mathematics, University of Wisconsin, Madison, WI 53706}
\email{ellenber@math.wisc.edu}
\urladdr{http://www.math.wisc.edu/~ellenber/}

\author{Dion Gijswijt}
\address{Delft Institute of Applied Mathematics, Delft University of Technology}
\email{d.c.gijswijt@tudelft.nl}
\urladdr{http://homepage.tudelft.nl/64a8q/}

\thanks{The first author is supported by NSF Grant DMS-1402620 and a Guggenheim Fellowship.  We thank Terry Tao, Tim Gowers, and Seva Lev for useful discussions during the production of this paper.}

\date{27 May 2016}

\begin{document}

\begin{abstract}
In this note, we show that the method of Croot, Lev, and Pach can be used to bound the size of a subset of $\F_q^n$ with no three terms in arithmetic progression by $c^n$ with $c < q$.  For $q=3$, the problem of finding the largest subset of $\F_3^n$ with no three terms in arithmetic progression is called the {\em cap problem}.  Previously the best known upper bound for the affine cap problem, due to Bateman and Katz~\cite{bateman}, was on order $n^{-1-\epsilon} 3^n$.
\end{abstract}

\maketitle

The problem of finding large subsets of an abelian group $G$ with no three-term arithmetic progression, or of finding upper bounds for the size of such a subset, has a long history in number theory.  The most intense attention has centered on the cases where $G$ is a cyclic group $\Z/N\Z$ or a vector space $(\Z/3\Z)^n$, which are in some sense the extreme situations.  We denote by $r_3(G)$ the maximal size of a subset of $G$ with no three-term arithmetic progression.  The fact that $r_3((\Z/3\Z)^n)$ is $o(3^n)$ was first proved by Brown and Buhler~\cite{brownbuhler}, which was improved to $O(3^n/n)$ by Meshulam~\cite{meshulam}. The best known upper bound, $O(3^n/n^{1+\epsilon})$, is due to Bateman and Katz~\cite{bateman}. The best lower bound, by contrast, is around $2.2^n$~\cite{edel}. 

The problem of arithmetic progressions in $(\Z/3\Z)^n$ has sometimes been seen as a model for the corresponding problem in $\Z/N\Z$.  We know (for instance, by a construction of Behrend~\cite{behrend}) that $r_3(\Z/N\Z)$ grows more quickly than $N^{1-\epsilon}$ for every $\epsilon > 0$.  Thus it is natural to ask whether $r_3((\Z/3\Z)^n)$ grows more quickly than $(3-\epsilon)^n$ for every $\epsilon > 0$.  In general, there has been no consensus on what the answer to this question should be.

In the present paper we settle the question, proving that for all odd primes $p$, $r_3((\Z/p\Z)^n)^{1/n}$ is bounded away from $p$ as $n$ grows.  

The main tool used here is the polynomial method, in particular the use of the polynomial method developed in the breakthrough paper of Croot, Lev, and Pach~\cite{CLP}, which drastically improved the best known upper bounds for $r_3((\Z/4\Z)^n)$.  In this case, they show that a subset of $G$ with no three-term arithmetic progression has size at most $c^n$ for some $c < 4$.  In the present paper, we show that the ideas of their paper can be extended to vector spaces over a general finite field.

\begin{rem}  The ideas of this paper were developed independently and essentially simultaneously by the two authors.  Since the arguments of our two papers were essentially identical, we present them as joint work.
\end{rem}

We begin with a slight generalization of Lemma 1 of \cite{CLP}.
Let $\F_q$ be a finite field and let $n$ be a positive integer. Let $M_n$ be the set of monomials in $x_1,\ldots, x_n$ whose degree in each variable is at most $q-1$, and let $S_n$ be the $\F_q$-vector space they span. 

Observe that the evaluation map $e: S_n\to\F_q^{\F_q^n}$ given by $e(p):=(p(a))_{a\in\F_q^n}$ is a linear isomorphism. Indeed, both spaces have dimension $q^n$ and the map $e$ is surjective since for every $a\in \F_q^n$ the polynomial $\prod_{i=1}^n (1-(x_i-a_i)^{q-1})$ is mapped to the indicator function of point $a$. 

For any real number $d$ in $[0,2n]$, let $M_n^d$ be the set of monomials in $M_n$ of degree at most $d$ and $S_n^d$ the subspace of $S_n$ they span. Write $m_d$ for the dimension of $S_n^d$.  By a slight abuse of notation, we use ``polynomial of degree at most $d$'' to mean an element of $S_n^d$.

\begin{prop}  Let $\F_q$ be a finite field and let $A$ be a subset of $\F_q^n$. Let $\alpha, \beta, \gamma$ be three elements of $\F_q$ which sum to $0$.

Suppose $P \in S_n^d$ satisfies $P(\alpha a+\beta b) = 0$ for every pair $a,b$ of distinct elements of $A$.  Then the number of $a \in A$ for which $P(-\gamma a) \neq 0$ is at most $2 m_{d/2}$.

\label{pr:clp}
\end{prop}

\begin{rem}
The proof of Proposition~\ref{pr:clp} is essentially the same as that of Lemma 1 of Croot-Lev-Pach~\cite{CLP}, which proves the proposition in the case $(\alpha,\beta,\gamma) = (1,-1,0)$.  In the $\gamma=0$ case, the conclusion of the proposition is that $P(0) = 0$ once $n > 2 m_{d/2}$; it turns out to be essential for the present application to have the added flexibility of forcing $P$ to take vanish at a larger set of places.
\end{rem}

\begin{proof}

\medskip

Any $P \in S_n^d$ is a linear combination of monomials of degree at most $d$, so we can write
\begin{equation}
\label{pxy}
P(\alpha x+\beta y) = \sum_{m,m' \in M_n^d\,:\,\deg(mm') \leq d} c_{m,m'} m(x) m'(y). 
\end{equation}

In each summand of \eqref{pxy}, at least one of $m$ and $m'$ has degree at most $d/2$.  We can therefore write (not necessarily uniquely)
\beq
P(\alpha x+\beta y) = \sum_{m \in M_n^{d/2}} m(x) F_m(y) + \sum_{m \in M_n^{d/2}} m(y) G_{m}(x)
\eeq
for some families of polynomials $F_m,G_m$ indexed by $m \in M_n^{d/2}$.

Now let $B$ be the $A \times A$ matrix whose $a,b$ entry is $P(\alpha a + \beta b)$.  Then
\beq
B_{ab} =  \sum_{m \in M_n^{d/2}} m(a) F_m(b) + \sum_{m \in M_n^{d/2}} G_{m}(a) m(b)
\eeq
This is an expression of $B$ as a sum of $2m_{d/2}$ matrices, each one of which visibly has rank $1$.  Thus the rank of $B$ is at most $2 m_{d/2}$.

On the other hand, our hypothesis on $P$ forces $B$ to be a diagonal matrix.  The bound on the rank of $B$ now implies that at most $2 m_{d/2}$ of the diagonal entries of $B$ are nonzero.  This completes the proof.

\end{proof}

\begin{thm} Let $\alpha, \beta, \gamma$ be elements of $\F_q$ such that $\alpha + \beta + \gamma = 0$ and $\gamma \neq 0$, and let $A$ be a subset of $\F_q^n$ such that the equation
\beq
\alpha a_1 + \beta a_2 + \gamma a_3=0
\eeq
has no solutions $(a_1,a_2,a_3) \in A^3$ apart from those with $a_1=a_2=a_3$.  As above, let $m_d$ be the number of monomials in $x_1, \ldots, x_n$ with total degree at most $d$ and in which each variable appears with degree at most $q-1$.

Then $|A| \leq 3m_{(q-1)n/3}$.
\label{th:main}
\end{thm}

\begin{proof}
Let $d$ be an integer in $[0,(q-1)n]$.  The space $V$ of polynomials in $S_n^d$ vanishing on the complement of $-\gamma A$ has dimension at least $m_d - q^n + |A|$.  Write $\SSS(A)$ for the set of all elements of $\F_q$ of the form $\alpha a_1 + \beta a_2$, with $a_1$ and $a_2$ distinct elements of $A$.  Then $\SSS(A)$ is disjoint from $-\gamma A$ by hypothesis, so any $P$ vanishing on the complement of $-\gamma A$ vanishes on $\SSS(A)$.  By Proposition~\ref{pr:clp}, we know that $P(-\gamma a)$ is nonzero for at most $2m_{d/2}$ points $a$ of $A$, for every $P$ in $V$.

View the elements of $V$ as functions on $\F_q^n$ and let $P\in V$ have maximal support. Let $\Sigma:=\{a\in \F_q^n : P(a)\neq 0\}$ be the support of $P$. We have $|\Sigma|\geq \dim V$ for otherwise, there would exist a nonzero $Q\in V$ vanishing on $\Sigma$. But then the support of $P+Q$ would strictly contain $\Sigma$, contradicting the choice of $P$. 

Since the support of $P$ is contained in $-\gamma A$, we have $|\Sigma|\leq 2m_{d/2}$ by Proposition \ref{pr:clp}, and hence $\dim V \leq  2m_{d/2}$. It follows that

\beq
m_d - q^n + |A| \leq 2m_{d/2}
\eeq
whence
\beq
|A| \leq 2m_{d/2} + (q^n - m_d).
\eeq

We note that $q^n - m_d$ is the number of $q$-power-free monomials whose degree is {\em greater} than $d$; these are naturally in bijection with those monomials whose degree is less than $(q-1)n-d$, of which there are at most $m_{(q-1)n - d}$.

Taking $d = 2(q-1)n/3$, we thus have
\beq
|A| \leq 2m_{(q-1)n/3} + (q^n - m_{2(q-1)n/3}) \leq 3m_{(q-1)n/3}
\eeq
as claimed.
\end{proof}

It is not hard to check that $m_{(q-1)n/3} / q^n$ is exponentially small as $n$ grows with $q$ fixed.  We can be more precise.  Let $X$ be a variable which takes values $0,1,\ldots,q-1$ with probability $1/q$ each; then $m_{(q-1)n/3} / q^n$ is the probability that $n$ independent copies of $X$ have mean at most $(q-1)/3$.  This is an example of a large deviation problem.  By Cram\'{e}r's theorem~\cite[\S 2.4]{timo}, we have
\beq
\lim_{n \ra \infty} \frac{1}{n} \log (m_{(q-1)n/3} / q^n) = -I((q-1)/3)
\eeq

where $I$ is the {\em rate} function of the random variable $X$, calculated as follows:  $I(x)$ is the supremum, over all $\theta$ in $\R$, of
\begin{equation}
\theta x - \log((1 + e^\theta + \ldots + e^{(q-1)\theta})/q).
\label{eq:sup}
\end{equation}
We note that \eqref{eq:sup} takes the value $0$ at $\theta=0$ and has nonzero derivative at $\theta=0$ unless $x=(q-1)/2$, so the supremum of \eqref{eq:sup} is positive; this shows $m_{(q-1)n/3} = O(c^n)$ for some $c < q$.

\begin{cor}  Let $A$ be a subset of $(\Z/3\Z)^n$ containing no three-term arithmetic progression.  Then $|A| = o(2.756^n)$.
\end{cor}

\begin{proof} Taking $q=3$ and $x=2/3$, the supremum in (\ref{eq:sup}) is attained when $e^\theta=(\sqrt{33}-1)/8$ and we obtain the bound $3e^{-I(2/3)} < 2.756.$  The theorem now follows by applying Theorem~\ref{th:main} with $\alpha = \beta = \gamma = 1$.
\end{proof}

\end{document}